\documentclass{article}
\usepackage{graphicx} 
\usepackage{amsmath,amssymb,amsthm,txfonts,amsbsy}
\usepackage{authblk}
\usepackage{graphicx}
\usepackage{color}
\usepackage{comment}
\usepackage[graph,color]{xy}
\usepackage{tikz, framed}
\usetikzlibrary{arrows,shapes,matrix,decorations.pathmorphing}
\usetikzlibrary{backgrounds}

\newcommand\qbin[3]{\left[\begin{matrix} #1 \\ #2 \end{matrix} \right]_{#3}}
\theoremstyle{plain} 
\newtheorem{Theorem}{Theorem} 

\newtheorem{Definition}[Theorem]{Definition}

\newtheorem{Lemma}[Theorem]{Lemma}
\newtheorem{Proposition}[Theorem]{Proposition}
\newtheorem{Corollary}[Theorem]{Corollary}

\theoremstyle{definition}
\newtheorem{Example}[Theorem]{Example}

\theoremstyle{remark} 
\newtheorem{examplex}[Theorem]{Example}
\newenvironment{example}
  {\pushQED{\qed}\examplex}
  {\popQED\endexamplex}

\newtheorem{Remark}[Theorem]{Remark}

 \newcommand{\FF}{\mathbb{F}}

\DeclareMathOperator{\rk}{rk}

\DeclareMathOperator{\Rsupp}{Rsupp}

\title{The incidence matrix of a $q$-ary graph}
\author[1]{Michela Ceria}
\affil[1]{\small 
Dept. of Mechanics, Mathematics \& Management, Politecnico di Bari,  Via Orabona 4 - 70125 Bari - Italy; michela.ceria@gmail.com}

\author[2]{Relinde Jurrius}
\affil[2]{\small Faculty of Military Sciences, Netherlands Defence Academy, The Netherlands; rpmj.jurrius@mindef.nl}
 
\date{}

\begin{document}

\maketitle
\begin{abstract}
  In this preprint we discuss a definition of a $q$-ary graph. Furthermore, we describe how to make an incidence matrix for it, with an eye on the corresponding $q$-matroid.  
\end{abstract}

\section{Introduction}

Ever since the re-discovery of $q$-matroids it has been an open question what the corresponding $q$-analogue of a graph is. The idea that vertices are $1$-dimensional spaces and edges are $2$-dimensional spaces, is relatively straightforward, but does not answer how to relate this to $q$-matroids. This paper will also not directly answer this question, but we will do this via the detour of the incidence matrix of a $q$-ary graph. An important role in our proposed definition is inspired by the notion of a $k$-regular $q$-ary graph, as defined in \cite{braun2024q} in order to study the $q$-analogue of a strongly regular graphs. We will define the incidence matrix of a $q$-ary graph with $q$-matroids in mind. \\

This manuscript is meant as a draft, an therefore lacks several preliminaries and references. We refer the reader to \cite{oxley2011matroid} for more information on matroids and specifically to Chapter 5 for matroids coming from graphs. For completeness we will shortly recall $q$-matroids and their representation.

\begin{Definition}\label{qmat}
A $q$-matroid is a pair $(E,r)$ with $E$ a finite dimensional vector space and $r:\{\text{subspaces of }E\}\to\mathbb{N}_0$ a function, the \emph{rank function}, with for all $A,B\subseteq E$:
\begin{itemize}
\item[(r1)] $0\leq r(A)\leq \dim A$
\item[(r2)] If $A\subseteq B$ then $r(A)\leq r(B)$.
\item[(r3)] $r(A  + B)+r(A\cap B)\leq r(A)+r(B)$ (semimodular)
\end{itemize}
\end{Definition}

An important class of matroids are the representable matroids.

\begin{Definition}
Let $E=\mathbb{F}_q^n$ and $G$ be a matrix over $\mathbb{F}_{q^m}$. Let $A\subseteq E$ and $Y$ a matrix whose row space is $A$. Then $r(A)=\rk(GY^\mathsf{T})$ is the rank function of a $q$-matroid: we call $G$ a \emph{representation} of the $q$-matroid.
\end{Definition}

A representation of a $q$-matroid is not unique: we will go in details later. Geometrically speaking, representable matroids are in bijection with $q$-systems, the $q$-analogue of projective systems. See \cite{randrianarisoa2020geometric,alfarano2022linear}.

\section{Defining a $q$-ary graph}

We propose the following definition of a $q$-ary graph. It is inspired by \cite{braun2024q}, where $k$-regular $q$-ary graphs are defined.

\begin{Definition}\label{def:q-graph2}
The vertices of our $q$-ary graph are the subspaces of dimension $1$ of $\FF_q^v$. The edges are some subspaces of dimension $2$. Furthermore, the edges have to satisfy the following: for every vertex $X$ and all edges adjacent to it (this is the neighbourhood of this vertex), it holds that the union of all these edges is equal to a subspace of $\FF_q^v$. Calling the degree of this subspace $d+1$, we say that $X$ has degree $\delta(X):=d$.
\end{Definition}

To avoid degenerate situations, we ask that the vertices of degree at least one do not all lie in a proper subspace of $\mathbb{F}^v$. (In that case, we take a smaller space over which we define the $q$-ary graph.) Note that we are ambiguous with the 1-dimensional subspaces of $\mathbb{F}^v$ that have degree 0, i.e., are not adjacent to any edge. These are isolated vertices that we might include if we want all vertices to form a subspace, but in practice we can just ignore them. (As in the classical case, where adding isolated vertices does not really give anything interesting.)

This definition is in fact equivalent to the following, that introduces coordinates and hence will be more useful in our proofs.

\begin{Definition}
Let $V=\mathbb{F}_q^v$ and let $E$ be a set of $2$-dimensional subspaces of $V$, the \emph{edges}. Then $(V,E)$ is a \emph{$q$-ary graph} if for all $c_1,c_2\in\mathbb{F}_q$ the \emph{$q$-graph property} holds: If $\langle\mathbf{x},\mathbf{y}_1\rangle$ and $\langle\mathbf{x},\mathbf{y}_2\rangle$ are (adjacent) edges, then $\langle\mathbf{x},c_1\mathbf{y}_1+c_2\mathbf{y}_2\rangle$ is also an edge.
\end{Definition}

There are a few elementary lemmas in graph theory that now have a nice $q$-analogue. First, we can relate the degrees and number of edges of a $q$-ary graph. This is a straightforward $q$-analogue of the classical result that the sum of the degrees of all vertices of a graph is equal to twice the number of edges.

\begin{Lemma}\label{DegEdg}
The following holds for the vertices and edges in a $q$-ary graph.
\[ \sum_{v\in V} \qbin{\delta(v)}{1}{q}=(q+1)\cdot \#\text{edges}. \]
\end{Lemma}

We now introduce some elementary definitions.

\begin{Definition}\label{walkcyc}
A \emph{walk} in a $q$-graph is a sequence $[v_0,e_1,v_1,\ldots,v_{\ell-1},e_\ell,v_\ell]$ of vertices $v_i$ and edges $e_i$ such that every edge $e_i$ is incident with the vertices $v_{i-1}$ and $v_i$. If {\em all vertices and all edges} are {\em distinct}, we call this walk a \emph{path}. If all edges are distinct and $v_0=v_\ell$, it is called a \emph{cycle}.
\end{Definition}

Note that for graphs, we ask for a path only for all vertices to be distinct, and this implies all edges are distinct. For $q$-graphs, this implication does not hold. If $v_0,v_1,v_2$ are vertices on the same edge $e$, we do not want $[v_0,e,v_1,e,v_2]$ to be a path. Because if it was, the next part of the definition makes $[v_0,e,v_1,e,v_2,e,v_0]$ a cycle.

\begin{Definition}\label{connected}
A $q$-ary graph is \emph{connected} if there is a path between every two vertices.
\end{Definition}

\begin{Definition}\label{subgraph}
A \emph{subgraph} of the $q$-ary graph $G$ is a $q$-ary graph whose edges are a subset of the set of edges of $G$ and that satisfies the $q$-graph property.
\end{Definition}

\begin{Definition}\label{ForestTree}
A subgraph that does not contain a cycle is called a \emph{forest}. If it is connected, it is called a \emph{tree}.
\end{Definition}

The next result is a $q$-analogue of the classical result that the number of vertices and edges in a tree differ by 1.

\begin{Lemma}\label{edgevert}
The following holds for a $q$-ary tree.
\[ \#\text{edges}\cdot q +1 = \#\text{vertices}. \]
\end{Lemma}
\begin{proof}
We do this by inductively adding edges to our tree. With one edge, the formula clearly holds, because there are $q+1$ vertices on one edge. When adding an extra edge, it will intersect the existing graph in exactly one vertex, otherwise it will not be connected or will create a cycle. So, there is one edge and $(q+1)-1$ vertices added to the graph, and the formula still holds.
\end{proof}

We make some examples of $q$-ary graphs.

\begin{example}\label{RoofTriang}
We consider the $q$-analogues of some very small graphs: the path $P_2$ (two edges) 
and the triangle $C_3$ (three edges). 
Because the neighbourhood of a vertex needs to be a space, we find that the $q$-analogue of $P_2$ is in $\mathbb{F}_q^3$ and consists of all edges (2-dimensional spaces) through one vertex (1-dimensional space). Since this is a tree, Lemma \ref{edgevert} holds: the number of vertices is equal to 1 (the common vertex) plus $q$ other vertices for each edge. To see that Lemma \ref{DegEdg} holds, note that there is one vertex of degree 2 and all other vertices have degree 1. \\
The $q$-analogue of $C_3$ has as edges all 2-dimensional spaces of $\mathbb{F}_q^3$. It has $\frac{q^3-1}{q-1}$ vertices that are all of degree 2, and $\frac{q^3-1}{q-1}$ edges. Since $\qbin{2}{1}{q}=q+1$, we see Lemma \ref{DegEdg} holds.
\end{example}

\begin{example}\label{P4C4}
We make two examples of trees in $\mathbb{F}_2^5$, that are the $q$-analogues of $P_4$, the path of length $4$, and $S_4$, the star with 4 rays. For $P_4$ we have the following vertices:
\[ \langle e_1,e_2\rangle, \langle e_2,e_3\rangle, \langle e_3,e_4\rangle, \langle e_4,e_5\rangle. \]
The vertices $e_2$, $e_3$ and $e_4$ now need another edge in order to make this a proper $q$-ary graph. This gives
\[ \langle e_2,e_1+e_3\rangle, \langle e_3,e_2+e_4\rangle, \langle e_4,e_3+e_5\rangle. \]
Together, they for a $q$-ary graph. Most vertices have degree $1$, except for $e_2$, $e_3$ and $e_4$ that have degree 2.
For $S_4$ we do something similar. The center of the star is $e_1$, then we get the following edges.
\[ \langle e_1,e_2\rangle, \langle e_1,e_3\rangle, \langle e_1,e_4\rangle, \langle e_1,e_5\rangle.
\]
We now need to add all other 2-dimensional spaces through $e_1$ as an edge. So in total there are 15 edges. All vertices in this graph have degree 1, except for $e_1$ that has degree $4$.
Since both graphs are trees, they satisfy Lemma \ref{edgevert}. For $P_4$, we get that $7\cdot 2+1=15$. Not all 1-dimensional spaces of 
$\mathbb{F}_2^5$ are vertices of the graph. For $S_4$ we see that $15\cdot 2+1=31$, which are all 1-dimensional subspaces of $\mathbb{F}_2^5$. 

Both these trees are maximal in $\mathbb{F}_2^5$, in the sense that they are not subgraphs of trees with a larger number of edges. However, they do not have the same number of vertices and edges.
\end{example}

We point out the next result about how to interpret the $q$-ary graph property in a geometric way.

\begin{Lemma}\label{dDim}
Consider all edges through a point $X$ of degree $d$. Then for each edge we can find a basis $\langle\mathbf{x},\mathbf{y}_i\rangle$, such that $X=\langle\mathbf{x}\rangle$ and the $\langle\mathbf{y}_i\rangle$ are exactly all $1$-dimensional subspaces of a certain $d$-dimensional space.
\end{Lemma}
\begin{proof}
From the definition we know that $X$ and all vertices adjacent to it are exactly all $1$-dimensional subspaces of a $(d+1)$-dimension space, the neighbourhood of $X$. Take a $d$-dimensional subspace $D$ inside this neighbourhood that does not contain $X$. Now there is a bijection between $1$-dimensional subspaces of $D$ and the edges through $X$. Indeed, every edge through $X$ intersects $D$ in a $1$-dimensional space and every edge intersects $D$ in a different $1$-dimensional subspace. Moreover, every $1$-dimensional subspace of $D$, together with $X$, spans an edge through $X$ by the $q$-ary graph property. Denoting a $1$-dimensional space of $D$ by $\langle\mathbf{y}_i\rangle$, we conclude that each edge can be written in the form $\langle\mathbf{x},\mathbf{y}_i\rangle$ with $\mathbf{y}_i\in D$.
\end{proof}

\section{The incidence matrix}

Our goal is to make an incidence matrix for a 
$q$-ary graph. In order to motivate this definition, we first have a closer look on how to make an incidence matrix of a graph, and write this process in a way that leads to the $q$-analogue we are looking for.

An incidence matrix of a graph has rows indexed by the vertices and columns indexed by the edges. An entry is $1$ if the vertex corresponding to the row and edge corresponding to the column are incident, otherwise it is zero. Every edge thus gets ``represented'' by a column, which is a vector in $\mathbb{F}_2$. This vector has Hamming weight 2, and its support corresponds to the set of vertices the edge is incident with.

From a matroid point of view, we could have also chosen to make the incidence matrix over $\mathbb{F}_3$. In that case, one choses an arbitrary direction for every edge. For the representation of an edge (i.e., column of the incidence matrix) this results in a $-1$ and $+1$ on the positions indexed by the initial and final vertex of the edge. Again, we have a vector of Hamming weight 2, and its support corresponds to the set of vertices the edge is incidence with. Furthermore, in order to make sure that every column has exactly one $+1$ and one $-1$, we ask that the column is orthogonal to the all-one vector: this is a vector of full Hamming weight.

This process can be extended to other fields. It is well-known that matroids that are representable over $\mathbb{F}_2$ and $\mathbb{F}_3$ are representable over any field. This representation can be found in a similar matter as before: make an incidence matrix of the graph where every edge corresponds to a column of the incidence matrix. This column vector needs to have Hamming weight 2, its support should correspond to the set of vertices the edge is incident with, and the vector needs to be orthogonal to the all-one vector. This is the formulation that inspires the $q$-analogue.

The incidence matrix of a $q$-ary graph in $\mathbb{F}_q^v$ will be a matrix over the field extension $\mathbb{F}_{q^m}$. Again every edge gets represented by a column. This is going to be a vector in $\mathbb{F}_{q^m}^v$. The vector has rank weight 2 and its support is the subspace corresponding to the edge. Furthermore, the vector is orthogonal to a fixed full-weight vector $\mathbf{u}$. Since all these entries need to be linearly independent as elements of $\mathbb{F}_{q^m}/\mathbb{F}_q$, we need that $m\geq v$. In fact, we take $m=v$ from now on. In our examples, we will often take $\mathbf{u}=[1,\alpha,\alpha^2,\ldots,\alpha^{m-1}]^\mathsf{T}$, where $\alpha$ is a primitive element of the field extension.

Before continuing to the incidence matrix, we focus on how the columns are determined.

\begin{Definition}
Let $A$ be an edge of a $q$-ary graph in $\mathbb{F}_q^m$ and let $\mathbf{u}$ be a full-weight vector in $\mathbb{F}_{q^m}^m$. A \emph{representation with respect to $\mathbf{u}$} for $A$ is a vector $\mathbf{v}\in\mathbb{F}_{q^m}^m$ satisfying $\Rsupp(\mathbf{v})=A$ and $\mathbf{v}\cdot\mathbf{u}=0$.
\end{Definition}

\begin{Remark}\label{AlgebraicallyIndependent}
Note that $\mathbb{F}_q^m$ and $\mathbb{F}_{q^m}$ are isomorphic as vector spaces over $\mathbb{F}_q$. This isomorphism depends on the choice of a basis of the field extension $\mathbb{F}_{q^m}/\mathbb{F}_q$, and the entries of the full-weight vector $\mathbf{u}$ can be taken as this basis. Then  $\mathbf{u}$ plays an important role in this isomorphism: given $\mathbf{x}\in\mathbb{F}_q^m$, the isomorphism is given by $\mathbf{x}\cdot\mathbf{u}$. In particular, this implies that vectors $\mathbf{x}_1,\ldots,\mathbf{x}_s$ are linearly independent over $\mathbb{F}_q$ if and only if $(\mathbf{u}\cdot\mathbf{x}_1),\ldots,(\mathbf{u}\cdot\mathbf{x}_s)$ are linearly independent over $\mathbb{F}_q$. We will use this fact extensively.
\end{Remark}

If it is clear from the context, we will omit the dependence on $\mathbf{u}$ when talking about representations. There are options for representations. In fact, if $\mathbf{v}$ is a representation, then $\alpha\mathbf{v}$ is as well. We prove some elementary properties of representations.


\begin{Lemma}
If $\mathbf{v}_1$ and $\mathbf{v}_2$ are representations of the same edge $A$, then for all $\alpha_1,\alpha_2\in\mathbb{F}_{q^m}$ such that $\alpha_1\mathbf{v}_1+\alpha_2\mathbf{v}_2\neq\mathbf{0}$ we have that $\alpha_1\mathbf{v}_1+\alpha_2\mathbf{v}_2$ is a representation of $A$ as well.
\end{Lemma}
\begin{proof}
First of all, note that $\alpha_1\mathbf{v}_1+\alpha_2\mathbf{v}_2$ is orthogonal to $\mathbf{u}$. Also, multiplication by $\alpha$ does not change the rank support of a vector (see \cite[Proposition 2.3]{jurrius2017defining}). We have left to show that $\Rsupp(\mathbf{v}_1+\mathbf{v}_2)=A$. By the same paper, we have that $\Rsupp(\mathbf{v}_1+\mathbf{v}_2)\subseteq\Rsupp(\mathbf{v}_1)+\Rsupp(\mathbf{v}_2)=A+A=A$. Now $\Rsupp(\mathbf{v}_1+\mathbf{v}_2)$ cannot have dimension 0 unless $\mathbf{v}_1=-\mathbf{v}_2$. It can also not have dimension $1$, because if $\mathbf{v}_1+\mathbf{v}_2$ has rank weight 1, there would be a nonzero vector over $\mathbb{F}_q$ that is orthogonal to $\mathbf{u}$, and this can not happen because the entries of $\mathbf{u}$ are elements of $\mathbb{F}_{q^m}$ that are linearly independent over $\mathbb{F}_q$. We conclude that $\Rsupp(\mathbf{v}_1+\mathbf{v}_2)$ has dimension $2$ and is thus equal to $\Rsupp(\mathbf{v}_1)+\Rsupp(\mathbf{v}_2)$.
\end{proof}

\begin{Proposition}\label{UniqueUpToAlpha}
A representation of an edge is unique up to multiplication by $\alpha\in\mathbb{F}_{q^m}^*$.
\end{Proposition}
\begin{proof}
The representation of an edge has rank weight 2, that means that there are two non-zero entries $\alpha_1$ and $\alpha_2$ such that the other entries are linear combinations of those over $\mathbb{F}_q$. Moreover, the representation has to be orthogonal to $\mathbf{u}$. So in order to find a representation, we have to find non-zero $\alpha_1$ and $\alpha_2$ that satisfy a linear equation over $\mathbb{F}_{q^m}$. This can be solved by taking any element of $\mathbb{F}_{q^m}^*$ as $\alpha_1$, and then $\alpha_2$ is fixed. So we have $q^m-1$ possibilities for a representation of a given edge. Since we have seen that $\Rsupp(\mathbf{v})=\Rsupp(\alpha\mathbf{v})$, it follows that these representations are $\alpha$-multiples of each other.
\end{proof}

Recall that we want to make an incidence matrix where every column is the representation of an edge of the $q$-ary graph. But we cannot take any representation we like for each edge: we want to take into account the $q$-ary graph property somehow. This leads us to the following definition.

\begin{Definition} \label{def:compatib}
Take a connected subset of edges of a $q$-ary graph, and consider their representations with respect to a fixed full-weight vector $\mathbf{u}$. We call this set of representations a \emph{compatible set of representations} if the following holds. For any representations $\mathbf{v}_1,\ldots,\mathbf{v}_s$ of edges adjacent to the same vertex, we can find vectors $\mathbf{x},\mathbf{y}_1,\ldots,\mathbf{y}_s$ such that 
\begin{itemize}
\item the common vertex is $\langle\mathbf{x}\rangle$;
\item $\mathbf{v}_i$ is a representation of the edge $\langle\mathbf{x},\mathbf{y}_i \rangle$;
\item for any $\lambda_1,\ldots,\lambda_s\in\mathbb{F}_q$ and the vector $\mathbf{v}:=\lambda_1\mathbf{v}_1+\lambda_2\mathbf{v}_2+\cdots+\lambda_s\mathbf{v}_s$ one of the following holds:
\begin{itemize}
\item $\mathbf{v}=\mathbf{0}$ and $\sum_{i=1}^s\lambda_i{\bf y}_i=\mathbf{0}$, or
\item $\mathbf{v}\neq\mathbf{0}$ and $\sum_{i=1}^s\lambda_i{\bf y}_i\neq\mathbf{0}$, moreover, $\mathbf{v}$ is a representation of the edge $\langle {\bf x}, \sum_{i=1}^s\lambda_i{\bf y}_i\rangle$.
\end{itemize}
\end{itemize}
\end{Definition}

In the last point, the space $\langle {\bf x}, \sum_{i=1}^s\lambda_i{\bf y}_i\rangle$ is always an edge because of the $q$-ary graph property. It can either be an edge of which we already have a representation in our set, or not. The last point in this definition reflects the $q$-ary graph property.

We can slightly extend our definition to edges that are not connected. Remember that if we have a graph of $c$ connected components, we can always make it into a connected graph by adding $c-1$ edges.

\begin{Definition}
Take a subset of edges of a $q$-ary graph, and consider their representations. We call this set of representations a \emph{compatible set of representations} if all connected components have compatible representations and moreover for any minimal set of edges needed to make a connected set, we can find representations such that the total of representations form a compatible set.
\end{Definition}

The following is direct from the definitions.

\begin{Lemma}
Any subset of a set of compatible representations is again a set of compatible representations.
\end{Lemma}
\begin{proof}
The statement follows taking the $\lambda_i$'s corresponding to the vectors that are not in the subspace as equal to zero. 
\end{proof}

In a set of compatible representations, we can multiply a representation by a scalar in $\mathbb{F}_q^*$ and the set stays compatible. However, multiplication by a scalar in $\mathbb{F}_{q^m}\backslash\mathbb{F}_q$ does not result in another compatible set of representations.

\begin{Lemma}\label{MultBySmallScalar}
Suppose $\mathbf{v}_1,\ldots,\mathbf{v}_s$ is a compatible set of representations of edges adjacent to the same vertex. Then for any $\mu_1,\ldots,\mu_s\in\mathbb{F}_q^*$ the set $\mu_1\mathbf{v}_1,\ldots,\mu_s\mathbf{v}_s$ is again a compatible set of representations.
\end{Lemma}
\begin{proof}
We will show that the choice of basis vectors $\mathbf{x},\mu_1\mathbf{y}_1,\ldots,\mu_s\mathbf{y}_s$ satisfies the properties given in Definition \ref{def:compatib}. For $i=1,\ldots,s$ we have that 
\[ \Rsupp(\mu_i\mathbf{v}_i)=\Rsupp(\mathbf{v}_i)=\langle\mathbf{x},\mathbf{y}_i\rangle=\langle\mathbf{x},\mu_i\mathbf{y}_i\rangle \]
and $\mu_i\mathbf{v}_i\cdot\mathbf{u}=0$, so $\mu_i\mathbf{v}_i$ and $\mathbf{v}_i$ are a representation of the same edge.

Now since $\mathbf{v}_1,\ldots,\mathbf{v}_s$ is a compatible representation, we have for any $\lambda_1,\ldots,\lambda_s\in\mathbb{F}_q$ that $(\lambda_1\mu_1)\mathbf{v}_1+\cdots+(\lambda_s\mu_s)\mathbf{v}_s=\mathbf{0}$ if and only if $(\lambda_1\mu_1)\mathbf{y}_1+\cdots+(\lambda_s\mu_s)\mathbf{y}_s=\mathbf{0}$, hence $\lambda_1(\mu_1\mathbf{v}_1)+\cdots+\lambda_s(\mu_s\mathbf{v}_s)=\mathbf{0}$ if and only if $\lambda_1(\mu_1\mathbf{y}_1)+\cdots+\lambda_s(\mu_s\mathbf{y}_s)=\mathbf{0}$. We have left to show that if the vector $\lambda_1(\mu_1\mathbf{v}_1)+\cdots+\lambda_s(\mu_s\mathbf{v}_s)$ is nonzero, it is a representation of the edge $\langle\mathbf{x},\lambda_1(\mu_1\mathbf{y}_1)+\cdots+\lambda_s(\mu_s\mathbf{y}_s)\rangle$. This follows again directly from the fact that $\mathbf{v}_1,\ldots,\mathbf{v}_s$ is a compatible representation: then if the vector $(\lambda_1\mu_1)\mathbf{v}_1+\cdots+(\lambda_s\mu_s)\mathbf{v}_s$ is nonzero, is a representation of the edge $\langle\mathbf{x},(\lambda_1\mu_1)\mathbf{y}_1+\cdots+(\lambda_s\mu_s)\mathbf{y}_s\rangle$, which is the same statement. We conclude that $\mu_1\mathbf{v}_1,\ldots,\mu_s\mathbf{v}_s$ is a compatible set of representations.
\end{proof}

This proof does not hold if we take $\mu_1,\ldots,\mu_s\in\mathbb{F}_{q^m}^*$, since we cannot multiply the vectors $\mathbf{y}_i$ by an element of $\mathbb{F}_{q^m}^*$. The next example shows that the statement is in fact false for $\mu_1,\ldots,\mu_s\in\mathbb{F}_{q^m}^*$.

\begin{Example}
Consider the $q$-analogue of $P_2$, as introduced in Example \ref{RoofTriang}, over $\mathbb{F}_2$. This $q$-ary graph has $3$ edges, all adjacent to the same vertex.
Let $\alpha$ be a primitive element of $\mathbb{F}_8$ with $\alpha^3+\alpha+1=0$ and take $\mathbf{u}=\left[ \begin{array}{ccc} 1 & \alpha & \alpha^2 \end{array}\right]^\mathsf{T}$. We claim the following set of representations is a compatible one:
\[ \mathbf{v}_1=\left[\begin{array}{c} \alpha \\ 1 \\ 0 \end{array}\right],\quad \mathbf{v}_2=\left[\begin{array}{c} 0 \\ \alpha^2 \\ \alpha \end{array}\right],\quad \mathbf{v}_3=\left[\begin{array}{c} \alpha \\ \alpha^6 \\ \alpha \end{array}\right]. \]
We need to find $\mathbf{x},\mathbf{y}_1,\mathbf{y}_2,\mathbf{y}_3$ as in the definition. Take $\mathbf{x}=\mathbf{e}_2$, $\mathbf{y}_1=\mathbf{e}_1$, $\mathbf{y}_2=\mathbf{e}_3$, and $\mathbf{y}_3=\mathbf{e}_1+\mathbf{e}_3$. It is clear that the vectors $\mathbf{v}_1, \mathbf{v}_2,\mathbf{v}_3$ have rank support $\langle\mathbf{e}_2,\mathbf{e}_1\rangle$, $\langle\mathbf{e}_2,\mathbf{e}_3\rangle$ and $\langle\mathbf{e}_2,\mathbf{e}_1+\mathbf{e}_3\rangle$, respectively.

Now we have to consider all linear combinations $\lambda_1\mathbf{v}_1+\lambda_2\mathbf{v}_2+\lambda_3\mathbf{v}_3$ with $\lambda_i\in\mathbb{F}_2$. It is straightforward to check that $\lambda_1\mathbf{v}_1+\lambda_2\mathbf{v}_2+\lambda_3\mathbf{v}_3=\mathbf{0}$ if and only if $\lambda_1=\lambda_2=\lambda_3$, and this is the same for the linear combination $\lambda_1\mathbf{y}_1+\lambda_2\mathbf{y}_2+\lambda_3\mathbf{y}_3=\mathbf{0}$. Assume $\lambda_1\mathbf{v}_1+\lambda_2\mathbf{v}_2+\lambda_3\mathbf{v}_3\neq\mathbf{0}$. If only one of the $\lambda$'s is nonzero, it is clear that $\lambda_1\mathbf{v}_1+\lambda_2\mathbf{v}_2+\lambda_3\mathbf{v}_3=\mathbf{v}_i$ is a representation of $\langle\mathbf{x},\mathbf{y}_i\rangle$. Assume now two out of three $\lambda$'s are nonzero, for example $\lambda_1=\lambda_3=1$ (the other cases go similarly). Then $\lambda_1\mathbf{v}_1+\lambda_3\mathbf{v}_3=\mathbf{v}_2$ and also $\lambda_1\mathbf{y}_1+\lambda_3\mathbf{y}_3=\mathbf{y}_2$. We have seen that $\mathbf{v}_2$ is a representation of $\langle\mathbf{x},\mathbf{y}_2\rangle$. We conclude that $\mathbf{v}_1,\mathbf{v}_2,\mathbf{v}_3$ is a compatible set of representations.

Now suppose we would replace $\mathbf{v}_2$ by $\alpha\mathbf{v}_2$. Note that $\Rsupp(\alpha\mathbf{v}_2)=\Rsupp(\mathbf{v}_2)$ so we are still considering the same edges. Take $\lambda_1=\lambda_2=1$ and $\lambda_3=0$. Then we get $\mathbf{v}=\left[ \begin{array}{ccc} \alpha & \alpha & \alpha^2 \end{array}\right]^\mathsf{T}$. The rank support of this vector is $\langle \mathbf{e}_1+\mathbf{e}_2,\mathbf{e}_3\rangle$. This space does not even contain $\mathbf{e}_2$, which is the common vertex of the edges represented by $\mathbf{v}_1$, $\alpha\mathbf{v}_2$ and $\mathbf{v}_3$. Hence $\mathbf{v}_1,\alpha\mathbf{v}_2,\mathbf{v}_3$ is not a compatible set of representations.
\end{Example}

However, we can multiply compatible representations by an element of $\mathbb{F}_{q^m}^*$, as long as we multiply all representations by the same scalar.

\begin{Lemma}\label{MultByBigScalar}
Suppose $\mathbf{v}_1,\ldots,\mathbf{v}_s$ is a compatible set of representations w.r.t. some fixed full-weight vector $\mathbf{u}$ of edges adjacent to the same vertex. Then for any $\mu\in\mathbb{F}_{q^m}^*$ the set $\mu\mathbf{v}_1,\ldots,\mu\mathbf{v}_s$ is again a compatible set of representations, but with respect to the full-weight vector $\mu\mathbf{u}$.
\end{Lemma}
\begin{proof}
First of all, note that $\mu\mathbf{u}$ is indeed a full-weight vector. We will show that the choice of basis vectors $\mathbf{x},\mathbf{y}_1,\ldots,\mathbf{y}_s$ satisfies the properties given in Definition \ref{def:compatib}, but with respect to $\mu\mathbf{u}$ instead of $\mathbf{u}$. For $i=1,\ldots,s$ we have that 
\[ \Rsupp(\mu\mathbf{v}_i)=\Rsupp(\mathbf{v}_i)=\langle\mathbf{x},\mathbf{y}_i\rangle \]
and $\mu\mathbf{v}_i\cdot\mu\mathbf{u}=0$, so $\mu\mathbf{v}_i$ and $\mathbf{v}_i$ are a representation of the same edge with respect to, respectively, $\mu\mathbf{u}$ and $\mathbf{u}$.

Now since $\mathbf{v}_1,\ldots,\mathbf{v}_s$ is a compatible representation, we have for any $\lambda_1,\ldots,\lambda_s\in\mathbb{F}_q$ that $\lambda_1\mathbf{v}_1+\cdots+\lambda_s\mathbf{v}_s=\mathbf{0}$ if and only if $\lambda_1\mathbf{y}_1+\cdots+\lambda_s\mathbf{y}_s=\mathbf{0}$. We have left to show that if the vector $\lambda_1(\mu\mathbf{v}_1)+\cdots+\lambda_s(\mu\mathbf{v}_s)$ is nonzero, it is a representation with respect to $\mu\mathbf{u}$ of the edge $\langle\mathbf{x},\lambda_1\mathbf{y}_1+\cdots+\lambda_s\mathbf{y}_s\rangle$. This follows again directly from the fact that $\mathbf{v}_1,\ldots,\mathbf{v}_s$ is a compatible representation with respect to $\mathbf{u}$: by Proposition \ref{UniqueUpToAlpha}, $\lambda_1(\mu\mathbf{v}_1)+\cdots+\lambda_s(\mu\mathbf{v}_s)=\mu(\lambda_1\mathbf{v}_1+\cdots+\lambda_s\mathbf{v}_s)$ has the same rank support as $\lambda_1\mathbf{v}_1+\cdots+\lambda_s\mathbf{v}_s$, which is the edge $\langle\mathbf{x},\lambda_1\mathbf{y}_1+\cdots+\lambda_s\mathbf{y}_s\rangle$. Finally, we note that $\mathbf{u}\cdot(\lambda_1(\mu\mathbf{v}_1)+\cdots+\lambda_s(\mu\mathbf{v}_s))=0=\mu\mathbf{u}\cdot(\lambda_1 \mathbf{v}_1+\cdots+\lambda_s\mathbf{v}_s)$.

We conclude that $\mu\mathbf{v}_1,\ldots,\mu\mathbf{v}_s$ is a compatible set of representations with respect to $\mu\mathbf{u}$.
\end{proof}

We have seen so far that if given a compatible representation, we can find a basis for the edges. But this is not a very constructive definition: how do we find a compatible representation? This question is answered in the next theorem.

\begin{Theorem}\label{THE_repres}
Consider $s$ edges through the same vertex $X=\langle\mathbf{x}\rangle$. Let $N$ be the subspace spanned by all vertices on these edges. Let $\mathbf{u}$ be a fixed full-weight vector in $\mathbb{F}_{q^m}^m$. By Lemma \ref{dDim}, we can assume there exist $\mathbf{y}_1,\ldots,\mathbf{y}_s$ in a a codimension-one space of $N$ not containing $X$ such that these edges can be written as $\langle\mathbf{x},\mathbf{y}_1\rangle,\ldots,\langle\mathbf{x},\mathbf{y}_s\rangle$.
Let $\mathbf{v}_i=-(\mathbf{u}\cdot\mathbf{y}_i)\mathbf{x}+(\mathbf{u}\cdot\mathbf{x})\mathbf{y}_i$. Then $\mathbf{v}_1,\ldots,\mathbf{v}_s$ is a compatible representation with respect to $\mathbf{u}$ for the edges $\langle\mathbf{x},\mathbf{y}_1\rangle,\ldots,\langle\mathbf{x},\mathbf{y}_s\rangle$.
\end{Theorem}
\begin{proof}
First of all, note that $\mathbf{v}_i$ cannot be zero. Since the entries of $\mathbf{u}$ are elements of $\mathbb{F}_{q^m}$ that are linearly independent over $\mathbb{F}_q$, the dot product of $\mathbf{u}$ with any nonzero vector in $\mathbb{F}_q^m$ will be nonzero. Combining this with the fact that $\mathbf{x}$ and $\mathbf{y}_i$ are linearly independent over $\mathbb{F}_q^m$ hence over $\mathbb{F}_{q^m}^m$, we find that $\mathbf{v}_i$ is nonzero.

We will prove that $\mathbf{v}_i$ is indeed a representation for $\langle\mathbf{x},\mathbf{y}_i\rangle$, that is, $\Rsupp(\mathbf{v}_i)=\langle\mathbf{x},\mathbf{y}_i\rangle$ and $\mathbf{v}_i\cdot\mathbf{u}=0$. The latter is straightforward:
\begin{align*}
\mathbf{v}_i\cdot\mathbf{u} & = \left(-(\mathbf{u}\cdot\mathbf{y}_i)\mathbf{x}+(\mathbf{u}\cdot\mathbf{x})\mathbf{y}_i\right)\cdot\mathbf{u} \\
 & = -(\mathbf{u}\cdot\mathbf{y}_i)(\mathbf{x}\cdot\mathbf{u})+(\mathbf{u}\cdot\mathbf{x})(\mathbf{y}_i\cdot\mathbf{u}) \\
 & = 0.
\end{align*}
Now for any vector in $\mathbb{F}_q^m$, we have that its rank support is equal to its span. Therefore, $\langle\mathbf{x},\mathbf{y}_i\rangle=\Rsupp(\mathbf{x})+\Rsupp(\mathbf{y}_i)$. Then, by our favourite lemma about rank support, we have that
\[ \Rsupp\left(-(\mathbf{u}\cdot\mathbf{y}_i)\mathbf{x}+(\mathbf{u}\cdot\mathbf{x})\mathbf{y}_i\right)\subseteq\Rsupp(\mathbf{x})+\Rsupp(\mathbf{y}_i). \]
To prove equality, we show that the left hand side has dimension 2. It is clearly at most $2$, and it is not $0$ since $\mathbf{v}_i\neq\mathbf{0}$, so it suffices to prove the dimension is not $1$. A vector can only have one-dimensional rank support if it is a $\mathbb{F}_{q^m}$-multiple of a vector in $\mathbb{F}_q^m$.

Suppose, towards a contradiction, that there is an element $\bar\alpha \in \FF_{q^m}$ such that $\bar\alpha {\bf v}_i \in \FF_q^m$.
This can happen in two cases: either both $\bar\alpha ({\bf u} \cdot {\bf x})$ and 
$\bar\alpha ({\bf u} \cdot {\bf y}_i)$ are in $\FF_q$ or there are components in $\FF_{q^m}$ in both
$\bar\alpha ({\bf u} \cdot {\bf x}){\bf y}_i$ and 
$\bar\alpha ({\bf u} \cdot {\bf y}_i){\bf x}$ that turn out to cancel.

In the first case, $\bar\alpha ({\bf u} \cdot {\bf x}) \in \FF_q$ and 
$\bar\alpha ({\bf u} \cdot {\bf y}_i) \in \FF_q$ means that $\bar\alpha ({\bf u} \cdot {\bf x})$ and $\bar\alpha ({\bf u} \cdot {\bf y}_i)$ are $\mathbb{F}_q$-multiples of each other, hence linearly dependent over $\mathbb{F}_q$. But then also ${\bf u} \cdot {\bf x}$ and ${\bf u} \cdot {\bf y}_i$ are dependent over $\mathbb{F}_q$. This is impossible, because ${\bf x}$ and ${\bf y}_i$ are linearly independent over $\mathbb{F}_q$ (see Remark \ref{AlgebraicallyIndependent}).

On the other side, in order to have cancellation we would need to have it component by component and this would again make ${\bf x}, {\bf y}_i$ linearly dependent. We conclude that $\mathbf{v}_i$ is indeed a representation for $\langle\mathbf{x},\mathbf{y}_i\rangle$.

We now show that $\mathbf{v}_1,\ldots,\mathbf{v}_s$ is a set of compatible representations. Take the vectors $\mathbf{x},\mathbf{y}_q,\ldots,\mathbf{y}_s$ as above, we show that they satisfy the definition of compatible representations. The first two properties are clear. For the third one, we write
\begin{align*}
\lefteqn{\lambda_1\mathbf{v}_1+\cdots+\lambda_s\mathbf{v}_s} \\
 & =\lambda_1(-(\mathbf{u}\cdot\mathbf{y}_1)\mathbf{x}+(\mathbf{u}\cdot\mathbf{x})\mathbf{y}_1)+\cdots+\lambda_s(-(\mathbf{u}\cdot\mathbf{y}_s)\mathbf{x}+(\mathbf{u}\cdot\mathbf{x})\mathbf{y}_s) \\
 & = -(\mathbf{u}\cdot(\lambda_1\mathbf{y}_1+\cdots+\lambda_s\mathbf{y}_s))\mathbf{x}+(\mathbf{u}\cdot\mathbf{x})(\lambda_1\mathbf{y}_1+\cdots+\lambda_s\mathbf{y}_s)
\end{align*}
Since $\mathbf{x}$ and $\lambda_1\mathbf{y}_1+\cdots+\lambda_s\mathbf{y}_s$ are linearly independent, by similar reasoning as before we see that $\lambda_1\mathbf{y}_1+\cdots+\lambda_s\mathbf{y}_s=\mathbf{0}$ iff $\lambda_1\mathbf{v}_1+\cdots+\lambda_s\mathbf{v}_s=\mathbf{0}$. Moreover, the last right hand side is exactly what we get when substituting $\lambda_1\mathbf{y}_1+\cdots+\lambda_s\mathbf{y}_s$ in the formula from the theorem. This implies $\mathbf{v}_1,\ldots,\mathbf{v}_s$ is indeed a set of compatible representations.
\end{proof}

This theorem only gives us a compatible set of representations for (all) edges adjacent to the same vertex. By repeating the procedure for all vertices, one edge can get different representations. The next Lemma \ref{ChangeBasis} shows that these representations will only differ up to a scalar of $\mathbb{F}_q^*$.

\begin{Lemma}\label{ChangeBasis}
Let $\langle \mathbf{x},\mathbf{y} \rangle$ be an edge. Then taking $\mathbf{x}$ and $\mathbf{y}$ as a basis, or taking $a_1 \mathbf{x} + b_1 \mathbf{y}$, and $a_2 \mathbf{x} + b_2 \mathbf{y}$ as a basis, gives representations w.r.t $\mathbf{u}$ that differ only by a scalar in $\mathbb{F}_q^*$.
\end{Lemma}
\begin{proof}
Fixing $\mathbf{u}$ and applying Theorem \ref{THE_repres} to the basis $\langle \mathbf{x},\mathbf{y} \rangle$ gives the representation $-(\mathbf{u}\cdot\mathbf{y})\mathbf{x}+(\mathbf{u}\cdot\mathbf{x})\mathbf{y}$. Applying the same Theorem to the basis $\langle a_1 \mathbf{x} + b_1 \mathbf{y}, a_2 \mathbf{x} + b_2 \mathbf{y} \rangle$ gives a representation that we can rewrite.
\begin{align*}
&-(\mathbf{u}\cdot(a_2 \mathbf{x} + b_2 \mathbf{y}))
(a_1\mathbf{x}+b_1\mathbf{y})
+
(\mathbf{u}\cdot(a_1 \mathbf{x} + b_1 \mathbf{y}))
(a_2\mathbf{x}+b_2\mathbf{y})= \\ &
-(a_2(\mathbf{u}\cdot\mathbf{x})+b_2(\mathbf{u}\cdot\mathbf{y})) (a_1\mathbf{x}+b_1\mathbf{y})+
(a_1(\mathbf{u}\cdot\mathbf{x})+b_1(\mathbf{u}\cdot\mathbf{y})) (a_2\mathbf{x}+b_2\mathbf{y})= \\ & 
(a_1b_2-a_2b_1) (-(\mathbf{u}\cdot\mathbf{y})\mathbf{x}+(\mathbf{u}\cdot\mathbf{x})\mathbf{y}).
\end{align*}
Note that $a_1b_2-a_2b_1$ is the determinant of the coordinate change between $\langle \mathbf{x},\mathbf{y} \rangle$ and $\langle a_1 \mathbf{x} + b_1 \mathbf{y}, a_2 \mathbf{x} + b_2 \mathbf{y} \rangle$, hence it is nonzero. Since $a_i,b_i \in \mathbb{F}_q$, $i=1,2$, the Lemma holds true.
\end{proof}

From this results it follows that we can use Theorem \ref{THE_repres} also to find a set of compatible representations for $q$-ary (sub)graphs that are not connected. Now we are ready to define the incidence matrix of a $q$-ary graph.

\begin{Definition}
The incidence matrix (with respect to a fixed full-weight vector $\mathbf{u}$) of a $q$-ary graph is a matrix with compatible representations (with respect to $\mathbf{u}$) as columns.
\end{Definition}

Note that this definition depends on the choice of $\mathbf{u}$. We have, however, the following (see Lemma \ref{MultBySmallScalar}):

\begin{Corollary}\label{FixReprv1FixAll}
The incidence matrix is defined up to ordering of the columns and multiplication of a column by an element of $\mathbb{F}_q^*$.
\end{Corollary}


\section{From an incidence matrix to a $q$-matroid}

We will prove that picking a different incidence matrix for a $q$-ary graph gives an isomorphic $q$-matroid. First, we state and prove the following result. It is already implicit in \cite[Lemma 21]{JP18}, but we prove it here for completeness.

\begin{Lemma}\label{RowAndColOperations}
Let $M$ be a $q$-matroid over $E=\mathbb{F}_q^n$ of rank $k$, represented by the matrix $G$ over $\mathbb{F}_{q^m}$. Then applying row operations over $\mathbb{F}_{q^m}$ to $G$ does not change $M$. Applying column operations over $\mathbb{F}_q$ to $G$ gives an isomorphic $q$-matroid $M'$.
\end{Lemma}
\begin{proof}
Applying row operations over $\mathbb{F}_{q^m}$ to $G$ means we multiply $G$ from the left with an invertible $k\times k$ matrix $A$ over $\mathbb{F}_{q^m}.$ For any $X\subseteq E$ of dimension $t$, we have that $r(X)=\rk(GY^\mathsf{T})$, where $Y$ is a $t\times n$ matrix with row space $X$. Since $\rk(GY^\mathsf{T})=\rk((AG)Y^\mathsf{T})$, we see that $G$ and $AG$ represent the same $q$-matroid. \\

When we apply column operations to $G$, it means we multiply from the right with an invertible $n\times n$ matrix $B$ over $\mathbb{F}_q$. Again for any subspace $X\subseteq E$ we have that
\[ r_{M(G)}(X)=\rk(GY^\mathsf{T})=\rk((GB)(B^{-1}Y^\mathsf{T}))=r_{M(GB)}(\varphi(X)) \]
where $\varphi$ is the isomorphism of $\mathbb{F}_q^n$ defined by multiplication with $B^{-1}$.
\end{proof}

We apply this result to show that the choice of incidence matrix of a $q$-ary graph changes the associated $q$-matroid at most up to equivalence.

\begin{Theorem}\label{DifferentIncSameMat}
    Associate to a $q$-ary graph an incidence matrix $G$. Let $M=M(G)$ be the $q$-matroid represented by $G$. Let $G'$ be another incidence matrix for the same $q$-ary graph. Then $M(G)$ and $M(G')$ are isomorphic $q$-matroids.
\end{Theorem}
\begin{proof}
The matrices $G$ and $G'$ can differ in several ways. We will argue that all of them lead to isomorphic $q$-matroids. By Lemma \ref{RowAndColOperations}, for a representable $q$-matroid, we can do row operations over $\mathbb{F}_{q^m}$ without changing the $q$-matroid. Column operations over $\mathbb{F}_q$ will lead to an isomorphic $q$-matroid. \\
Suppose $G$ and $G'$ are defined with respect to the same full-rank vector $\mathbf{u}$. Then we have from Corollary \ref{FixReprv1FixAll} that the incidence matrix is defined up to multiplying columns by an element of $\mathbb{F}_q^*$. This can be viewed as a column operation over $\mathbb{F}_q$ hence will yield an isomorphic $q$-matroid. Next, we note that the order in which the representations of edges appear in the incidence matrix, leads to isomorphic $q$-matroids, because re-ordering is a column operation over $\mathbb{F}_q$. \\
Now suppose $G$ and $G'$ are defined with respect to different full-wight vectors $\mathbf{u}$ and $\mathbf{u}'$, respectively. This can be viewed as a change of basis of the field extension $\mathbb{F}_{q^m}/\mathbb{F}_q$. Let $C$ be the full-rank $m\times m$ matrix over $\mathbb{F}_q$ such that $\mathbf{u}'=C\mathbf{u}$. If $C$ is the matrix of multiplication by some scalar $\mu\in\mathbb{F}_{q^m}^*$, so $\mathbf{u}'=\mu\mathbf{u}$, Lemma \ref{MultByBigScalar} gives that the columns of $G'$ are equal to the corresponding columns of $G$, multiplied with $\mu$. This can be seen as a row operation over $\mathbb{F}_{q^m}$, hence does not change the associated $q$-matroid. If $C$ is some other matrix, we can replace $\mu$ by $C$ everywhere in Lemma \ref{MultByBigScalar}, finding that $G'=CG$. Hence $G$ and $G'$ differ by left multiplication of a matrix over $\mathbb{F}_q$, thus over $\mathbb{F}_{q^m}$, and give the same $q$-matroid.
\end{proof}

\section{Concluding remarks}\label{Conclusion}

In this preprint we have given the definition of a $q$-ary graph. We have shown that we can associate an incidence matrix, and hence a representation of a $q$-matroid, to it. We have motivated why this definition of the incidence matrix is a $q$-analogue of the incidence matrix in the classical case. Further motivation can be found in viewing graphs as $q$-ary graphs over the field $\mathbb{F}_1$, but we will not go further into that here.

We realise that some of our proofs require the introduction of coordinates, and later proving that the results are independent of this choice of coordinates. This suggest that there is a ``bigger picture'' behind our algebraic approach, related to $q$-systems and linear sets.

The representation of an edge of a $q$-ary graph as a column vector, as defined in this paper, can be viewed as the $q$-analogue of an incidence vector of a set. This idea might be extended to other $q$-analogues, for example in defining the $q$-analogue of a matroid polytope. However, this needs more research and insights in the geometric picture behind our approach.

A big open question that we do not yet know how to answer, is how to go from a $q$-ary graph directly to a $q$-matroid. This requires to study the $q$-analogue of a set of edges. In the $q$-analogue, this should be a space. We hope to determine if it is a circuit in the $q$-matroid by looking at whether the corresponding edges, whatever that means, form a cycle in the $q$-ary graph. Also here we feel more research is needed.

\section*{Acknowledgements}
M. Ceria belongs to  the ``National
Group for Algebraic and Geometric Structures, and their Applications'' (GNSAGA - INdAM), which supported this work.
This work was supported as well by the Italian Ministry of University and Research under the Programme “Department of Excellence” Legge 232/2016 (Grant No. CUP - D93C23000100001).

\bibliographystyle{abbrv}
\bibliography{biblio} 
\end{document}